\documentclass{gtpart}

\usepackage[utf8]{inputenc}
\usepackage{plain}
\usepackage{booktabs}
\usepackage{amsmath, amssymb}
\usepackage{amsthm}
\usepackage{textcomp}
\usepackage{graphicx}
\usepackage{graphics}
\usepackage{import}
\usepackage{color}
\usepackage{float}
\usepackage{pdfpages}
\usepackage{changepage}
\usepackage{cleveref}
\usepackage{mathptmx}
\usepackage[english]{babel}
\usepackage{url}
\usepackage{hyperref}

\usepackage{array}
\usepackage{enumerate}
\usepackage{graphicx}
\usepackage{lmodern}
\usepackage{mathrsfs}
\usepackage{cite}
\usepackage{fancyhdr}

\theoremstyle{plain}
\newtheorem{theorem}{Theorem}
\newtheorem{lemma}{Lemma}
\newtheorem{proposition}{Proposition}

\theoremstyle{definition}
\newtheorem{definition}{Definition}
\newtheorem{example}{Example}

\theoremstyle{remark}
\newtheorem{remark}{Remark}

\newtheorem*{main*}{Main Theorem}
\newtheorem*{proposition*}{Proposition}
\newtheorem*{properties*}{Properties}

\newtheorem*{corollary*}{Corollary}
\newtheorem*{lemma*}{Lemma}
\newtheorem*{remark*}{Remark}
\newtheorem*{conj*}{Conjectures}
\newtheorem*{examples*}{Examples}
\newtheorem*{definition*}{Definition}
\def\loba{\loba}

\def\loba{\loba}
\def\loba{\hbox{\rm J\kern-1pt I}}

\def\gd{\hbox{$\bullet\kern-2pt$ --\kern-4pt---\kern-4pt--- 
\kern-2pt}}
\def\grinf{\hbox{$\bullet\kern-2pt $ ---\kern-4pt\raise
6pt\hbox{$\infty$}\kern-11pt --\kern-6pt ---\kern-5pt ---\kern+1pt}}
\def\gr #1{\hbox{$\bullet\kern-2pt $ ---\kern-3pt\raise
6pt\hbox{$#1$}\kern-8pt --\kern-6pt ---\kern-5pt ---\kern+1pt}}
\def\gra #1{\hbox{$\bullet\kern -2pt$ --\raise
8pt\hbox{$\alpha_{#1}$}\kern-13pt
---\kern-4pt\kern-4pt ---\kern-5pt ---\kern +2pt}}

\title[Salem numbers, spectral radii and growth rates of hyperbolic Coxeter 
groups]{Salem numbers, spectral radii and growth\\ rates of hyperbolic Coxeter groups}

\author{Ruth Kellerhals}
\givenname{Ruth}
\surname{Kellerhals}
\address{Department of Mathematics\\
University of Fribourg\\
CH-1700 Fribourg\\Switzerland}
\email{Ruth.Kellerhals@unifr.ch}

\author{Livio Liechti}
\givenname{Livio}
\surname{Liechti}
\address{Department of Mathematics\\
University of Fribourg\\
CH-1700 Fribourg\\Switzerland}
\email{livio.liechti@unifr.ch}


\begin{document}

\begin{abstract}    

We show that not every Salem number appears as the growth rate of a cocompact hyperbolic Coxeter group. We also 
give a new proof of the fact that the growth rates of planar
hyperbolic Coxeter groups are spectral radii of Coxeter transformations,  
and show that this need not be the case for growth rates of hyperbolic tetrahedral Coxeter groups.


\end{abstract}

\maketitle

\vskip1cm
%%%%%%%%%%%%%%%%%%%%%%%%%%%%%%%%%
{\it In memoriam Ernest B.  Vinberg}
%%%%%%%%%%%%%%%%%%%%%%%%%%%%%%%%%
\section{Introduction}\label{Intro}
%%%%%%
Let $P\subset\mathbb H^n$ be a compact hyperbolic Coxeter polyhedron of dimension $n\ge2$. This means
that $P$ is a convex polyhedron bounded by $N\ge n+1$ hyperbolic hyperplanes $H_1,\ldots,H_N$ that either intersect under a dihedral angle of the form $\pi/k$ for an integer $k\ge2$ or admit a common perpendicular in $\mathbb H^n$. The group generated by the reflections $r_i$
in the hyperplanes $H_i\,,\,1\le i\le N\,,$ is a discrete group $G\subset\mathrm{Isom}\mathbb H^n$ called a \emph{(cocompact) hyperbolic Coxeter group}. When $N$ is small, their quotient spaces $\mathbb H^n/G$ give rise to hyperbolic orbifolds and manifold covers enjoying particularly nice extremality properties.
The simplest examples of hyperbolic Coxeter groups arise from Coxeter $k$-gons $P=(p_1,\ldots,p_k)\subset\mathbb H^2\,,$ for $\,k\ge3\,,$ where 
the integers $p_1,\ldots,p_k\ge 2$ satisfy $\frac{1}{p_1}+\dots+\frac{1}{p_k}<k-2$.  In particular, there exist infinitely many non-isometric Coxeter polygons.  Furthermore, a well known result of Siegel states that the hyperbolic $2$-orbifold of minimal volume originates from the Coxeter triangle $(2,3,7)$. 

A hyperbolic Coxeter group is a geometric realisation of a certain abstract Coxeter system. A Coxeter system $(W,S)$ of rank $N$ consists of a group $W$ with finite generating set $S=\{s_1,\ldots,s_N\}$ and with relations 
%$s_i^2=1$ and 
$(s_is_j)^{m_{ij}} = 1$ for all indices~$i,j.$
Here the integers $m_{ij}$ satisfy $m_{ii}=1$  and $m_{ij}=m_{ji}\in\{2,3,\ldots,\infty\}$, otherwise.  An exponent $m_{ij}=\infty$ indicates that the product $s_is_j$ is of infinite order.
The growth series $f_S(t)$ of $(W,S)$ is given by 
\[
f_S(t)=1+\sum\limits_{k\ge 1}\,a_kt^ k\,,
\]
where $a_k\in \mathbb{ Z}$ equals the number of words $w\in W$ with 
$S$-length $k$, and it characterises the complexity of $W$.  By a result of Steinberg, $f_S(t)$ is a rational function that
depends on the set of finite subgroups of $W$. The inverse $\tau=1/R$ of the radius of convergence $R$ of $f_S(t)$
is a real algebraic integer called the \emph{growth rate} of $W$, or also the \emph{growth rate} of its Coxeter polyhedron.

For a hyperbolic Coxeter group $G\subset\mathrm{Isom}\mathbb H^n,$ the growth rate satisfies $\tau>1$ so that $G$ is of exponential growth. More specifically, for $G$ compact, results of Floyd, Plotnick and Parry imply that $\tau$ is a Salem number or a quadratic unit when $n=2$ or $n=3$.
Recall that a Salem number is a real algebraic integer $>1$ such that all Galois conjugates have absolute value not 
greater than 1 and at least one of them has absolute value equal to 1.
An interesting example is given by the growth rate of the reflection group associated with the Coxeter triangle  $(2,3,7)$. It is equal to Lehmer's number
$\alpha_L\approx1.176281$ with minimal polynomial 
$L(t)=t^{10}+t^9-t^7-t^6-t^5-t^4-t^3+t+1$. Observe that $\alpha_L$ is the smallest Salem number known to date.

In reverse, our first main result sheds light on the realisation of Salem numbers as growth rates of hyperbolic Coxeter groups and their polyhedra.
%%%%%
\begin{theorem}\label{main01}
Not every Salem number is the growth rate of a compact hyperbolic Coxeter polyhedron. 
\end{theorem}
%%%%%

Consider an abstract Coxeter system $(W,S)$ of rank $N$ together with its natural representation as a discrete group of reflections 
in $\mathrm{GL}_N(V)$ for a certain quadratic real vector space $V$. 
A Coxeter element  $c\in W$ is a word of $S$-length $N$ so that every generator in $S$ appears exactly once.
Its representative $C\in\mathrm{GL}_N(V)$ is called a \emph{Coxeter transformation}. By means of its order and the eigenvalue spectrum one can decide about the nature of $W$. For example, $W$ is finite if and only if $C$ is of finite order; see \cite{How}.

For integers $p_1,\ldots,p_k\ge 2$, consider the \emph{star graph} $\mathrm{Star}(p_1,\ldots,p_k)$
given by the tree with one vertex of valency $k$ that has $k$ outgoing paths of respective
lengths $p_i-1$.  Such a graph describes a Coxeter system $(W,S)$ as follows. Each node $s$ of the graph 
yields a generator in $S$, and the relations of $W=\langle S\rangle$ are defined by $s^2=1$, by
$(st)^3=1$ if the nodes $s,t$ are joined by an edge, and by $(st)^2=1$, otherwise.

Our second main result establishes a bridge between the growth rates of reflection groups
of planar hyperbolic Coxeter groups and the spectral radii of Coxeter transformations of suitably parametrised star graphs.
%%%%%
\begin{theorem}
\label{main02}
Let~$k\ge3$, $p_1,\dots,p_k\ge2$ be integers with $\frac{1}{p_1}+\dots+\frac{1}{p_k}<k-2$.
Then the growth rate of the reflection group given by the compact Coxeter $k$-gon $(p_1,\dots,p_k)$ in $\mathbb H^2$ equals the spectral radius of the Coxeter transformation of~$\mathrm{Star}(p_1,\dots,p_k)$.
\end{theorem}
%%%%%
This result is implicitly stated in the work of E. Hironaka~\cite{Eko03} and based on a connection to Alexander polynomials of pretzel links and the theory of fibered 
knots and links~\cite{Eko,Eko04}. Our method of proof of Theorem~\ref{main02} is different
and does not use any topology. Instead, we provide and exploit explicit recursion formulas.

Conversely, for dimensions $n>2$, we show that not every growth rate of a compact Coxeter polyhedron in $\mathbb H^n$
is equal to the spectral radius of a Coxeter transformation. In fact, in contrast to the Coxeter tetrahedron with Coxeter symbol $[4,3,5]$, the
growth rate of the Coxeter tetrahedron $[3,5,3]$ is not equal to the spectral of a Coxeter transformation.  Note that
the tetrahedron $[4,3,5]$ has minimal volume among all hyperbolic Coxeter tetrahedra while $[3,5,3]$ has minimal growth rate among all Coxeter polyhedra in $\mathbb H^3$.

%%%%%
The paper is organised as follows. In Section~\ref{section2}, we recall in a first part~\ref{section2-1} the essential concepts of Coxeter group, Coxeter graph and Coxeter transformation. In the second part~\ref{section2-2}, we describe hyperbolic Coxeter polyhedra and their associated reflection groups.  Some examples
provide a glimpse into the wealth of hyperbolic Coxeter groups which--in contrast to the spherical and affine 
Coxeter groups--are far from being classified; see \cite{V1} and \cite{F-web}, for example. In Section \ref{section3}, we review the basic notions of growth series and growth rate of a Coxeter system. The partial order on the set of Coxeter systems
and its implication for growth rates and minimality of certain Coxeter systems are presented. These aspects will be useful tools in some of our proofs. This first part \ref{section3-1} is completed by a brief discussion of the connection of certain growth rates with Salem numbers. In part \ref{section3-2}, we identify the planar hyperbolic Coxeter group having the second smallest growth rate
as a preparation to prove Theorem \ref{main01}. Section \ref{spectral} is devoted to the proof of Theorem \ref{main02}, and in Section \ref{growthspectral}, we show that not every growth rate of a hyperbolic Coxeter group is the spectral radius of a Coxeter transformation.
%%%%%%%%%%%

\subsection*{Acknowledgements} 
The authors would like to thank Yohei Komori for some helpful comments on an earlier draft of the article. 
The first author was partially supported by Schweizerischer Nationalfonds 200020-172583.

\newpage

%%%%%%%%%%%%%%%%
\section{Geometric Coxeter groups}\label{section2}
%%%%
\subsection{Coxeter groups and Coxeter elements}\label{section2-1}
%%%%
A \emph{Coxeter system} $(W,S)$ is a group $W$ with finite generating set $S=\{s_1,\dots,s_N\}$ 
and with the relations~$(s_is_j)^{m_{ij}} = 1$ for all indices~$i,j,$
where the integers $m_{ij}$ satisfy $m_{ii}=1$  and $m_{ij}=m_{ji}\in\{2,3,\ldots,\infty\}$, otherwise. 
Here, $m_{ij}=\infty$ means that the product $s_is_j$ is of infinite order. The group $W$ is a \emph{Coxeter group of rank~N}. 
Given a Coxeter system~$(W,S)$, the corresponding \emph{Coxeter diagram} $\Gamma$ is the weighted graph whose vertices~$v_1,\dots, v_N$
correspond to the generators~$s_1,\dots,s_N$, and an edge of weight~$m_{ij}$ joins~$v_i$ to~$v_j$ when~$m_{ij}\ge3$. 
Edge weights are typically omitted if they equal~$3$. 

Coxeter groups admit a canonical geometrical representation. Let~$V$ be a real vector space with basis~$e_1,\dots,e_N$, where the vector~$e_i$ corresponds to the generator~$s_i$ of~$S$, say. 
Equip $V$ with the symmetric bilinear form~$B$ defined by 
\begin{equation}\label{eq:product1}
B(e_i,e_j)=\left\{
\begin{array}{ll} -\cos(\pi/m_{ij})&m_{ij}<\infty\,;\\
-1&m_{ij}=\infty\,.\\
\end{array}
\right.
 \end{equation} 
The geometric representation~$\rho:W\to\mathrm{GL}_N(V)$ defined by~\[\rho(s_i)(v) = v - 2B(e_i,v)e_i\quad(v\in V)\]
associates to each generator $s_i$ the corresponding reflection $r_i=\rho(s_i)$
with respect to the subspace $H_i=\{v\in V\mid B(e_i,v)=0\}$. The map $\rho$ preserves the form~$B$ and is faithful with discrete image. 
In this context, it is not difficult to see that the group $W$ is finite if and only if the form $B$ is positive definite. Suppose that $(W,S)$ is irreducible, that is, its Coxeter diagram $\Gamma$ is connected. Then, $W$ is called \emph{spherical} or \emph{affine}
if the form $B$ is of signature $(N,0)$ or $(N-1,0)$, respectively.  The spherical and affine Coxeter groups are completely classified for arbitrarily large $N$ (see \cite{Hum} for details).

Let $(W,S)$ be a Coxeter system of rank $N$.  A word~$c\in W$ of $S$-length $N$ so that every generator in $S$ appears exactly once is called a \emph{Coxeter element}. If the Coxeter diagram $\Gamma$ is a tree, then the different Coxeter elements form a single conjugacy class
in $W$ by a result of Steinberg~\cite{Steinberg59}.
The image~$C=\rho(c)\in \mathrm{GL}_N(V)$ of a Coxeter element $c\in W$ is called a \emph{Coxeter transformation}. 
By means of its order and its eigenvalues one can decide about the nature of $W$; see \cite{AC76} and \cite{How}.
In particular, the Coxeter group $W$ is finite and spherical if and only if the order of $c$ is finite.
Consider the \emph{Coxeter adjacency matrix} $A=2\,I-2B$ of the Coxeter diagram $\Gamma$.
If~$\Gamma$ is a tree and~$\alpha$ is the leading eigenvalue of $A$, 
then the spectral radius~$\lambda$ of the Coxeter transformation $C$ satisfies the equation
\begin{equation}\label{eq:eigen-spectral}
\alpha^2 = 2 + \lambda+\lambda^{-1}\,.
\end{equation}
For a reference, see McMullen~\cite{McM02}.
We note that if all the weights of the Coxeter diagram $\Gamma$ are equal to~$3$, then the Coxeter adjacency matrix $A$ equals the 
adjacency matrix of the underlying abstract graph.
%%%%%%%%%%%%%%%%
\subsection{Hyperbolic Coxeter groups}\label{section2-2}
%%%%%
Let $\mathbb H^n$ denote the standard hyperbolic $n$-space in its linear model 
\[
\mathbb H^n=\{x\in\mathbb R^{n+1}\mid q_{n,1}(x)=x_1^2+\dots+x_n^2-x_{n+1}^2=-1\,,\,x_{n+1}>0\}\,.
\]
In this setting, a hyperbolic hyperplane $H$ is given by the intersection of $\mathbb H^n$ with the Lorentzian-orthogonal complement $e^L$ of a (space-like) vector $e\in\mathbb R^{n+1}$
normalised to be of norm $q_{n,1}(e)=1$. The reflection $r=r_H$ with respect to the hyperbolic hyperplane $H$ of $\mathbb H^n$ 
is given by $\,x\mapsto r(x)=x-2\,\langle e,x\rangle_{n,1}\,e\,$
where $\langle x,y\rangle_{n,1}$ denotes the bilinear form of signature $(n,1)$ associated with $q_{n,1}$.
The isometry group of $\mathbb H^n$ is given by the group $\mathrm{O}_\circ(n,1)$ of positive Lorentzian matrices; see \cite[Chapter 3]{Rat-book}. Notice that each isometry is a finite composition of reflections in hyperbolic hyperplanes. 

A \emph{Coxeter polyhedron} $P\subset\mathbb H^n$ is a convex polyhedron all of whose dihedral angles are of the form
$\pi/m$ for an integer $m\in\{2,3,\ldots,\infty\}$.  We always assume that $P$ is of finite volume and hence bounded by finitely many hyperplanes $H_1,\dots,H_N$ with $N\ge n+1$. 
Represent each hyperplane $H_i=e_i^L$ by a unit normal vector $e_i$
directed away from $P$ so that the half-space $H_i^-=\{x\in\mathbb H^n\mid \langle x,e_i\rangle_{n,1}\le 0\}$ contains~$P$. In particular, $P=\cap%\limits
_{i=1}^N H_i^-$.
For $i=1,\ldots,N$, the reflections $r_i$ with respect to $H_i$ generate a discrete group
$G=(G,S)$ of hyperbolic isometries with generating set $S=\{r_1,\ldots,r_N\}$.
The elements of $S$ satisfy $r_i^2=1$ and, for $i\not=j$, the rotation relations 
$\,(r_ir_j)^{m_{ij}} = 1$ if $m_{ij}<\infty$. In particular, the exponents $m_{ij}$ are symmetric with respect to $i,j$, and when finite, they are closely related to the dihedral angles formed by $H_i,H_j$ when intersecting in $\mathbb H^n$.
Products $r_ir_j$ of infinite order can be described in a geometric way as well (see below).
As a consequence, the group $G$ is  a Coxeter group which we call a \emph{hyperbolic Coxeter group}. 
Furthermore, the group $G$ can be described by a Coxeter diagram $\Gamma$ as above.

Consider the Gram matrix $\mathrm{Gram}(P)=\big(\langle e_i,e_j\rangle_{n,1}\big)$ whose entries are described as follows.
\begin{equation}
\label{eq:product2}
\langle e_i,e_j\rangle_{n,1}=\left\{
\begin{array}{ll}
-\cos(\pi/m_{ij})&\textrm{if $H_i,H_j$ intersect at $\pi/m_{ij}$ in $\mathbb H^n$};\\
-1&\textrm{if $H_i,H_j$ meet at $\partial\mathbb H^n$ forming the angle $0$};\\
-\cosh(l_{ij})&\textrm{if $H_i,H_j$ are at distance $l_{ij}>0$ in $\mathbb H^n$}.\\
\end{array}
\right.
\end{equation}
Many combinatorial and geometric features of $P$ can be read off from its Gram matrix $\mathrm{Gram}(P)$; see \cite{V1}. For example, the polyhedron $P$ is a compact $n$-simplex if $\mathrm{Gram}(P)$ is an indecomposable and invertible matrix of signature $(n,1)$ such that all principal submatrices are positive definite.  

In view of~(\ref{eq:product2}), a product $r_ir_j$ is of infinite order if the hyperplanes $H_i,H_j$ are (hyperbolic) parallel or at distance $l_{ij}$ in $\mathbb H^n$.
We take this additional information into account and describe the Coxeter polyhedron $P$ and its Coxeter group $G=(G,S)$ by means of their Coxeter diagram $\Gamma$ as follows. If $H_i,H_j$ meet at $\partial\mathbb H^n$, then we join the nodes $v_i,v_j$ by a bold edge (omitting the weight $\infty$); if $H_i,H_j$ are at distance $l_{ij}>0$ in $\mathbb H^n$, then $v_i,v_j$ are joined by a dotted edge (usually without the weight $l_{ij}$). In \cite[Theorem A]{FT1},
Felikson and Tumarkin showed that the Coxeter diagram of a compact Coxeter polyhedron in $\mathbb H^n$
has  always a dotted edge if $n\ge5$ (compare also with Example \ref{ex3} below).

In the case that the Coxeter polyhedron $P$ is bounded by only a few hyperplanes, its description by the {\it Coxeter symbol} is more convenient. For example, $[p_1,\ldots,p_k]$ or $[q_1,\ldots,q_l,\infty]$ with integer labels $p_i,q_j\ge3$ are associated with linear Coxeter diagrams
with $k$ or $l+1$ edges marked by the respective weights. 
The Coxeter symbol $[(p^k,q)]$ describes a cyclic Coxeter diagram with $k\ge1$ consecutive edge weights $p$ followed by the weight $q$; see \cite[Appendix]{JKRT1}, for example. 

\begin{example}\label{ex1}
Let $k\ge3$ and $p_1,\ldots,p_k\ge 2$ be integers.  A compact Coxeter $k$-gon $P=(p_1,\dots,p_k)$ 
with interior angles $\pi/p_1,\ldots,\pi/p_k$ exists in $\mathbb H^2$ if and only if its (normalised) angle sum satisfies 
$\frac{1}{p_1}+\dots+\frac{1}{p_k}<k-2$. 
\end{example}

\begin{example}\label{ex2}
As in the 2-dimensional case, there are infinitely many non-isometric compact Coxeter polyhedra in $\mathbb H^3$. Their description is not of the same elementary nature as in Example \ref{ex1} but there is a complete characterisation due to Andreev; see \cite{Roeder}, for example.
\end{example}

\begin{example}\label{ex3}
Compact Coxeter $n$-simplices were classified by Lann\'er and exist for $n\le 4$, only (see \cite[Part II, Chapter 5]{VinbergII}).  Of particular interest will be the Coxeter tetrahedra given by the symbols $[3,5,3]$ and $[4,3,5]$. Due to work of Koszul and Chein, non-compact Coxeter $n$-simplices are classified as well and exist for $n\le 9$.  All their volumes are computed in \cite{JKRT1}.
\end{example}

\begin {example}\label{ex4}
Compact Coxeter polyhedra with mutually intersecting bounding hyperplanes exist in $\mathbb H^n$ for $n\le 4$, only. Such a polyhedron is either a simplex or one of the seven Esselmann polyhedra. These results are due to Felikson and Tumarkin \cite[Theorem A]{FT1}.  For more detailed information, see \cite{F-web}.
\end{example}

%%%%%%%%%%%%%%%%
\section{Salem numbers and growth rates}\label{section3}
%%%%
\subsection{Growth series and growth rates}\label{section3-1}
%%%%
For a Coxeter system $(W,S)$ with generating set $S=\{s_1,\ldots,s_N\}$ we introduce the notion and review the relevant properties of the growth series of $(W,S)$; for references, see \cite{AC76}, \cite{K336} and \cite{KK}.
The growth series $f_S$ of $W$ is given by
\[
f_S(t)=\sum\limits_{w\in W}\,t^{l_S(w)}=1+\sum\limits_{k\ge 1}\,a_kt^ k\,,
\]
where $a_k\in \Bbb Z$ equals the number of words $w\in W$ with 
$S$-length $k$. 
By Steinberg's formula, %\cite{Steinberg},
\begin{equation}\label{eq:Steinberg}
\frac{1}{f_S(t^ {-1})}=\sum\limits_{{W_T<W\atop\scriptscriptstyle{{\vert W_T\vert<\infty}}}}\,
\frac{(-1)^{\vert T\vert}}{f_T(t)}\,,
\end{equation}
where $W_T\,,\,T\subset S\,,$ is a finite Coxeter subgroup of $W$, and where $W_{\varnothing}=\{1\}$.
By a result of Solomon, %\cite{Solomon}
the associated growth polynomials $f_T$ are given explicitly in terms of their exponents $n_1,\ldots,n_p$ according to
\begin{equation}\label{eq:Solomon}
f_T(t)=\prod\limits_{i=1}^ {p}\,[n_i+1]\,.
\end{equation}
Here we use the standard notations $\,[l]:=1+t+\dots+ t^{l-1}$. By replacing $t$ by $t^{-1}$, one gets $\,[l](t)=t^{l-1}[l](t^{-1})$.

For the exponents $\,n_1=1,n_2,\ldots,n_{p}\,$ of $G_T$, 
see \cite[Section 9.7]{CoMo}, for example. 
In particular, the dihedral group $D_2^ l\,,\,l\ge2\,,$ has exponents $\,1,l-1\,$ and growth polynomial $[2][l]$. 
As a consequence, in its disk of convergence, the growth series $f_S(t)$ of a Coxeter system $(W,S)$ is a rational function and quotient of 
coprime monic polynomials $p(t),q(t)\in\mathbb Z[t]$ of equal degree. The \emph{growth rate} $\tau_W=\tau_{(W,S)}$ is defined by
\[
\tau_{W}=\limsup\limits_{k\rightarrow\infty}\sqrt[k]{a_k}\,
\]
and equals the inverse of the radius of convergence $R$ of $f_S(t)$. 

Growth rates satisfy a nice monotonicity property on the partially ordered set of Coxeter systems. 
For two Coxeter systems $(W,S)$ and $(W',S')$, one declares $(W,S)\le(W',S')$ if there is an injective map $\iota:S\rightarrow S'$ such that $m_{st}\le m'_{\iota(s)\iota(t)}$ for all $s,t\in S$. If $\iota$ extends to an isomorphism between $W$ and $W'$, one writes $\,(W,S)\simeq(W',S')$, and $\,(W,S)<(W',S')$ otherwise. In the latter case, we often say that the Coxeter system $(W',S')$ \emph{dominates} the system $(W,S)$.
This partial order satisfies the descending chain condition since $m_{st}\in\{2,3,\ldots,\infty\}$ where $s\not=t$. In particular,
any strictly decreasing sequence of Coxeter systems is finite, which, in the hyperbolic case, leads to the notion of \emph{minimal hyperbolic Coxeter systems}; see \cite{McM02}.
In this context, we shall exploit the following result of Terragni \cite[Section 4.3]{Terragni}.
\begin{lemma}\label{Terragni}
If $\,(W,S)\le(W',S')$, then  $\tau_{(W,S)}\le\tau_{(W',S')}$.
\end{lemma}

\begin{example}\label{growth-monotonicity}
Instead of defining and ordering (abstract) Coxeter systems, we indicate their ordering on the level of Coxeter graphs.
Consider the Coxeter graphs ordered according 
%%%%%%%%%%%
\begin{figure}[H]
\centering
\begin{picture}(130,30)
    % Circles
    \put(0,15){\circle*{4}}
    \put(15,15){\circle*{4}}
    \put(30,15){\circle*{4}}
    % Line
    \put(0,15){\line(1,0){30}}
    
        % Circles
    \put(60,15){\circle*{4}}
    \put(75,15){\circle*{4}}
    \put(90,15){\circle*{4}}
    % Line
    \put(60,15){\line(1,0){30}}
    
            % Circles
    \put(120,22){\circle*{4}}
    \put(140,22){\circle*{4}}
    \put(130,7){\circle*{4}}
    % Line
    \put(120,22){\line(1,0){20}}
     \put(131,7){\line(3,5){10}}
     \put(129,7){\line(-3,5){10}}   

 %   \put(-60,10){$\Gamma\,\,:$}
 
    \put(40,12){$\le$}     
    \put(100,12){$\le$} 
       
    \put(20,18){$\scriptstyle 8$}  
    \put(78,19){$\scriptstyle \infty$}  
    \put(126,26){$\scriptstyle \infty$}        
\end{picture}
\caption{The naturally ordered Coxeter systems $[3,8]\,\le\,[3,\infty]\,\le\,[(3^2,\infty)]\,$.}
\label{fig:subgraph1}
\end{figure}
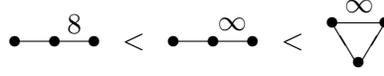
%%%%%%%%%%
Then, by Lemma \ref{Terragni}, we deduce that $\tau_{[3,8]}\le\tau_{[3,\infty]}\le\tau_{[(3^2,\infty)]}$.
\end{example}
%%%%%%%%%%%%%%%%
Consider a hyperbolic Coxeter group $G=(G,S)$ with set $S$ of generating 
reflections in hyperplanes of $\mathbb H^n$ as usual. 
By \cite[Corollary, p. 376]{CD1991}, if the Coxeter polyhedron $P$ of $G$ is compact, then the rational function $f(t)=f_S(t)$ is reciprocal (resp.\ anti-reciprocal) for $n$ even (resp.\ $n$ odd), that is,
\begin{equation}\label{eq:reciprocal}
f(t^ {-1})=\left\{
\begin{array}{ll}
 f(t)&\textrm{for $n\equiv0\,(2)$};\\
-f(t)&\textrm{for $n\equiv1\,(2)$}.\\
\end{array}
\right.
\end{equation}
As a consequence of (\ref{eq:Steinberg})--(\ref{eq:reciprocal}), the growth function of a compact hyperbolic Coxeter polygon $P=(p_1,\ldots,p_k)\subset\mathbb H^2$ can be determined as follows.
Since 
\begin{equation}\label{eq:f-planar}
\frac{1}{f(t)}=\frac{1}{f(t^{-1})}=1-\frac{k}{[2]}+\sum\limits_{i=1}^k\frac{1}{[2][p_i]}\,,
\end{equation}
one deduces that (see also \cite[Section 2]{Floyd})

\begin{equation}\label{eq:Delta}
f(t)=:\frac{[2]\,[p_1]\cdots[p_k]}{\Delta_{p_1,\ldots,p_k}(t)}=\frac{[2]\,[p_1]\cdots[p_k]}{[2]\,[p_1]\cdots[p_k]-\sum\limits_{i=1}^k\,t\,[p_1]\cdots[p_i-1]\cdots[p_k]}\,.
\end{equation}
In general, the growth rate $\tau_G=\tau_P=1/R$ of $G$ (and of $P$) is an algebraic integer which, 
by results of Milnor and de la Harpe, is always strictly bigger than 1. 
By results of Floyd, Plotnick and Parry (see also \cite{Komori}), the growth rate $\tau_P$ of a compact Coxeter polyhedron $P\subset\mathbb H^n$ with $n=2$ and $n=3$ is either a quadratic unit or a \emph{Salem number}, that is, $\tau_P$ is
a real algebraic integer $>1$ all of whose Galois conjugates have absolute value not greater than 1 and at least one of them has absolute value equal to 1.

The smallest Salem number known to date is Lehmer's number $\alpha_L\approx1.176281$ with minimal polynomial 
$L(t)=t^{10}+t^9-t^7-t^6-t^5-t^4-t^3+t+1$. By a result of E. Hironaka \cite{Eko} (see also \cite{KK}),
Lehmer's number $\alpha_L$ is the minimal growth rate~$\tau_1$ among all $\tau_P$ with 
$P$ a compact hyperbolic Coxeter polygon, and it is realised by the triangle $P=(2,3,7)$ in a unique way.
In this context, recall Siegel's result that the associated Coxeter group $[3,7]$ yields the (unique) minimal volume quotient space among all 
hyperbolic 2-orbifolds of finite volume; the second smallest hyperbolic 2-orbifold is given by the compact quotient space $\mathbb H^2/[3,8]$. In comparison with $\alpha_L$, the growth rate $\tau_2$ of the triangle group $[3,8]$ has minimal polynomial $t^{10}-t^7-t^5-t^3+1$ and is $\approx1.230391$.
By looking at the complete list $\mathrm{(L)}$ of all Salem numbers of degree $\le 44$, which is due to Boyd, 
Mossinghoff and others (for a survey, see \cite{Smyth}; for the list $\mathrm{(L)}$, see \cite{Moss}), the growth rate $\tau_2$
is the seventh smallest Salem number in the list $\mathrm{(L)}$, only.

For compact Coxeter polyhedra $P\subset \mathbb H^3$, the smallest growth rate  has been determined by Kellerhals and Kolpakov in \cite{KK}. It is realised by the Coxeter group $[3,5,3]$ in a unique way and of value $\approx1.350980$ with minimal polynomial $t^{10}-t^9-t^6+t^5-t^4-t+1$. In this way, $\tau_{[3,5,3]}$ is bigger than the first 47 smallest Salem numbers as listed in $\mathrm{(L)}$. 

\begin{remark}\label{435}
It is interesting to compare the compact Coxeter tetrahedra $[3,5,3]$ and $[4,3,5]$.  In contrast to the growth rate, the volume of $[4,3,5]$ is smaller than the one of $[3,5,3]$; see \cite[Appendix]{JKRT1}. However, the Coxeter diagram of $[3,5,3]$ has an internal symmetry, and by results of Martin and his co-authors (see \cite{Martin} and the references therein), 
the quotient of $\mathbb H^3$ by the $\mathbb Z_2$-extension of the group $[3,5,3]$ has smallest volume among \emph{all} hyperbolic 3-orbifolds.
\end{remark}

Note that for higher dimensional Coxeter polyhedra $P\subset\mathbb H^n,\,n\ge4\,,$ there are many examples 
whose growth rates are \emph{not} Salem numbers anymore. 
A simple example is given by the compact right-angled 
$120$-cell $\mathcal C\subset \mathbb H^4$ with ${\mathfrak f_0}=600$ vertices and ${\mathfrak f_3}=120$ dodecahedral facets. As a consequence of \cite[Proposition 3.2]{KP}, the growth function $f(t)$ of $\mathcal C$ is given by
\[
f(t)=
\frac{[2]^4}{t^4-116\,t^3+366\,t^2-116\,t+1}\,\,,
\]

whose denominator polynomial is irreducible over $\mathbb Z$ with two inversive pairs of positive real roots (see Section for more details). In particular, the growth rate $\tau_{\mathcal C}\approx112.763387$ is not a Salem number.
In \cite{Umemoto}, Umemoto constructed an infinite sequence of non-isometric $4$-dimensional compact Coxeter polyhedra whose growth rates are \emph{real $2$-Salem numbers}. These are algebraic integers $\alpha>1$ which have exactly 
one conjugate $\beta$ outside the closed unit disk, and at least one conjugate on the unit circle. Then all other conjugates of $\alpha$ different from $\alpha^{-1}$, $\beta$ and $\beta^{-1}$ lie on the unit circle.  As in the case of Salem numbers and their minimal polynomials, called \emph{Salem polynomials}, the minimal polynomial or \emph{2-Salem polynomial} of $\alpha$ is an irreducible palindromic polynomial of even degree.  

\begin{remark}\label{2Salem}
In general, it is a difficult problem to decide whether a palindromic monic polynomial $p(t)\in\mathbb Z[t]$ 
is irreducible. Specifically, for a palindromic monic polynomial with four simple roots that are positive real and the other roots on the unit circle, 
it is difficult to decide whether it is a 2-Salem polynomial or splits into two Salem polynomials and possibly cyclotomic polynomials over $\mathbb Z$.  
\newline
In \cite[Theorem 6.3, Theorem 7.1]{Cannon} (see also\cite[Theorem 2.12]{CaWag}), Cannon provides a necessary and sufficient condition for $p(t)$ to be a Salem polynomial, and he showed that the growth rate of a compact hyperbolic 4-simplex is \emph{not} a Salem number.
\end{remark}

\begin{remark}\label{Perron}
Salem numbers and real 2-Salem numbers are special \emph{Perron numbers}. 
A Perron number is a real algebraic integer $>1$ all of
whose conjugates are of strictly smaller absolute value. In \cite{KP}, Kellerhals and Perren formulate a conjecture 
which can be stated in a modified way as follows: \emph{For every $n\ge2$, the growth rate of a hyperbolic Coxeter $n$-polyhedron
is a Perron number}. By means of the software package CoxIter \cite{CoxIter} and its webversion, both due to Guglielmetti, one verifies  that the conjecture is true for all known hyperbolic Coxeter polyhedra of finite volume.
\end{remark}

%%%%%
\subsection{Not every Salem number appears as a growth rate}\label{section3-2}
%%%%%
With these preliminaries we are now ready to prove that not every Salem number is the growth rate of a hyperbolic Coxeter group.
As a first step, we consider hyperbolic Coxeter polygons whose growth rates are small Salem numbers and prove the following result. 
\begin{proposition}\label{238}
The second smallest growth rate 
$\tau_2$ of a compact Coxeter polygon in $\mathbb H^2$ is realised in a unique way by the triangle 
with Coxeter symbol $[3,8]$.
The Salem number $\tau_2\approx 1.230391$ has
minimal polynomial $t^{10}-t^7-t^5-t^3+1$ and is the seventh smallest Salem number in the list $\mathrm{(L)}$.
\end{proposition}

\begin{proof} The strategy of the proof is similar to the one given for $\tau_1$
in \cite[Section 4.1]{KK}.
By Steinberg's formula \ref{eq:Steinberg} (see also~(\ref{eq:Delta})), the growth function $f_{[3,8]}(t)$ of the Coxeter triangle
group $[3,8]$ equals
\[
f_{[3,8]}(t)=\frac{p(t)}{t^{10} - t^7 - t^5 - t^3 + 1}\,,
\]
with a certain numerator polynomial $p(t)\in\mathbb Z[t]$.

Let $P\subset\mathbb H^2$ be a compact Coxeter polygon with number of vertices
$\mathfrak f_0$ and
associated Coxeter group $G$. Denote by $\pi/k_v$ the interior
angle at the vertex $v$ in $P$. That is, the vertex stabiliser $G_v\subset G$ 
is the dihedral group $D_2^{k_v}$ of order $2k_v\ge 4$,
with growth polynomial $[2][k_v]$.  As a consequence,
\begin{equation}\label{eq:f-function}
\begin{aligned}
\frac{1}{f(t)}=&1-\frac{\vert S\vert}{[2]}+\sum\limits_{v\in P}\frac{1}{[2][k_v]}=1-\frac{1}{[2]}\,
\sum\limits_{v\in P}\,\big\{1-\frac{1}{[k_v]}\big\}\\
=&1-\frac{t}{[2]}\,\sum\limits_{v\in P}\,\frac{[k_v-1]}{[k_v]}=:1-\frac{t}{[2]}\,\sum\limits_{v\in P}\,h_{v}(t),\,\\
\end{aligned}
\end{equation}
where the help functions $h_v\,,\,v\in P,$ and their sum $H(t)$ can be written in the form 
\begin{equation}\label{eq:help}
h_{k}(t)=h_{k_v}(t)=\frac{[n(k)]}{[n(k)+1]}\quad,\quad H(t):=\sum\limits_{v\in P}h_k(t)\,,
\end{equation} 
since
the exponents of the group $D_2^{k_v}$ are equal to
$n_1=1$ and $n_2=n(k)=k_v-1$. 

By results of \cite[Section 3.1]{KK}, we have the following properties for the functions $h_k(t)$ and $[n]$ for all $x\in (0,1]$.

\begin{itemize}
 \item[(a)] For all $i<j$, $\quad 0<h_i(t)<h_j(t)<1\,.$

\item[(b)] For any positive integer $l$, 
$\,\displaystyle{\frac{2}{[2]}>\frac{[l]}{[l+1]}}\,.$
\end{itemize}

In order to show that the growth rate $\tau$ for any compact Coxeter polygon $P\subset \Bbb H^2$ which is
not isometric to a Coxeter triangle $[3,m]$ for $m=7,8$ satisfies $\tau>\tau_{[3,8]}$, 
it is sufficient to show 
that for each $t\in(0,1/\tau_{[3,8]}]$, the value $1/f_{[3,8]}(t)$ is strictly bigger than the corresponding 
value $1/f(t)$ for $P$. By the identities \eqref{eq:f-function} and \eqref{eq:help}, this means that we have to show that
\begin{equation}\label{eq:H-compare}
H(t)>H_{[3,8]}(t)=\frac{1}{[2]}+\frac{[2]}{[3]}+\frac{[7]}{[8]}\quad\hbox{for all}\quad t\in(0,1/\tau_{[3,8]}]\,.\end{equation}

To this end, 
we consider three cases in terms of the number of vertices $\mathfrak f_0\ge3$.

\smallskip\emph{Case 1.} Let $\mathfrak f_0\ge5$, and consider compact Coxeter polygons with at least five vertices. Here, all the interior angles 
may be equal to
$\pi/2$. Hence, by \eqref{eq:H-compare} and (b), we get the obvious
estimate
\[
H(t)\ge \frac{5}{[2]}>H_{[3,8]}(t)\,.
\]

\smallskip\emph{Case 2.} Let $\mathfrak f_0=4$, and consider
hyperbolic Coxeter quadrilaterals $P$ by noticing that they may have at most three right angles. 
Hence, by the properties (a) and (b) above, we get the estimate
\begin{equation}\label{eq:case2}
H(t)\ge \frac{3}{[2]}+\frac{[2]}{[3]}\,.
\end{equation}
Therefore, \eqref{eq:H-compare}, (b) and \eqref{eq:case2} imply that
a compact Coxeter quadrilateral has strictly bigger growth rate than the triangle
$[3,8]$. 

\smallskip\emph{Case 3.} 
Let $\mathfrak f_0=3$, and let $P$ be a Coxeter triangle with angular
existence condition
$1/p+1/q+1/r<1$ for integers $p,q,r\ge2$. Hence, at most one angle of $P$ can be equal to $\pi/2$, for example. 
The angular existence condition and the properties (a) and (b) imply that 
$H$ must satisfy at least one of the following inequalities.
\begin{subequations}
\begin{align} 
%\begin{split}
%\mathrm{(12a)}\hskip2cm 
H(t)&\ge \frac{1}{[2]}+\frac{[2]}{[3]}+\frac{[l]}{[l+1]}=:H_l(t)\quad\hbox{for}\quad l\ge 7\,;\hskip2cm&\label{eq:a}\\
%\mathrm{(12b)}\hskip2cm 
H(t)&\ge \frac{1}{[2]}+\frac{[3]}{[4]}+\frac{[4]}{[5]}=:H_4(t)>H_{[3,8]}(t)\,,&\label{eq:b}\\
%\mathrm{(12c)}\hskip2cm 
H(t)&\ge 2\,\frac{[2]}{[3]}+\frac{[3]}{[4]}=:H_{34}(t)>H_{[3,8]}(t)\,,&\label{eq:c}
%\end{split}
\end{align}
\end{subequations}
with equality in \eqref{eq:a}
%(12a) 
only if $l=7$ and therefore $G\cong [3,8]$.
Indeed, the first inequality holds for all Coxeter triangles
having angles $\pi/2,\pi/3$, and by comparison with \eqref{eq:H-compare}, the function $H_l(t)\,,\,l\ge7\,,$ does coincide
with $H_{[3,8]}(t)$ precisely for $l=7$. As for \eqref{eq:b}, which concerns right-angled Coxeter triangles with no angle equal to $\pi/3$,
we consider the difference function $\Delta_b(t):=H_4(t)-H_{[3,8]}(t)$
for $t\in(0,1]$. A straightforward computation yields
$$\Delta_b(t)=\frac{t^2\,F(t)}{[2][3][5]\,(t^2+1)\,(t^4+1)}\quad,$$
where $F(t)=t^{8}+t^7-t^5-t^4-t^3+t+1=t^7+(1-t^3)\,(1+t-t^5)$ so that $\Delta(t)>0\,$ on $\,(0,1]$.
Finally, for \eqref{eq:c}, and the comparison with Coxeter triangles with no angle equal to $\pi/2$, we study the difference function $\Delta_c(t):=H_{34}(t)-H_{[3,8]}(t)$
for $t\in(0,1]$. One easily checks that
$$\Delta_c(t)=\frac{t\,(1-t^3+t^6)}{[2][3]\,(t^2+1)\,(t^4+1)}>0\quad\hbox{for}\quad t\in(0,1]\,\,.$$
Therefore, 
$[3,8]$ has smallest growth rate among all compact planar Coxeter triangles different from $[3,7]$.
Hence, we proved our assertion.
\end{proof}

\begin{remark}\label{strict}
Consider the hyperbolic Coxeter triangles with Coxeter symbols $[3,8]$ and $[3,\infty]$, respectively. In contrast to $[3,8]$ the Coxeter triangle $[3,\infty]$ is not compact but still of finite area. By Steinberg's formula \eqref{eq:Steinberg}, the difference of their inverted growth functions satisfies
\begin{equation}\label{eq:38}
\frac{1}{f_{[3,8]}(t)}-\frac{1}{f_{[3,\infty]}(t)}=\frac{t^9}{[3][8]}>0\quad\hbox{for}\quad t>0\,\,.
\end{equation}

This fact shows that $\tau_2=\tau_{[3,8]}<\tau_{[3,\infty]}$ which sharpens the first part of the estimate given in Example \ref{growth-monotonicity}.
\end{remark}

We are now ready to prove our first main result as given by Theorem~\ref{main01} in Section~\ref{Intro}. It provides the answer
to a question which 
the first author raised at the Oberwolfach Mini-Workshop on {\it Reflection Groups in Negative Curvature} in April 2019.

\begin{proof}[Proof of Theorem~\ref{main01}]
Let $P\subset\mathbb H^n$ be a compact Coxeter polyhedron of dimension $n\ge2$ and denote by
$G=(G,S)$ its associated Coxeter group.
Let $\tau=\tau_P=\tau_G$ be the growth rate of $P$ and of $G$.  
For $n=2$, the smallest growth rate $\tau_1=\tau_{[3,7]}$ equals the smallest known Salem number given by Lehmer's number $\alpha_L$, and the second smallest growth rate $\tau_2=\tau_{[3,8]}$ is the seventh smallest Salem number by Proposition \ref{238}.
Hence the five Salem numbers strictly in between $\tau_1$ and $\tau_2$
in the list $\mathrm{(L)}$ of \cite{Moss} do \emph{not} appear as growth rates of compact Coxeter polygons in $\mathbb H^2$. This holds, for example, for the fifth smallest Salem number $\approx1.216391$ with minimal polynomial $t^{10}-t^6-t^5-t^4+1$.

For $n=3$, we know by \cite[Section 3]{KK} that the minimal growth rate is the Salem number with minimal polynomial
$t^{10}-t^9-t^6+t^5-t^4-t+1$ and belongs to the Coxeter tetrahedron $[3,5,3]$. Since $\tau_{[3,5,3]}\approx 1.350980$,
it follows that none of the first 47 smallest Salem numbers as listed in (L) appear as growth rates of compact Coxeter polyhedra in $\mathbb H^3$.

Suppose that $n\ge4$. Our strategy is to show that the growth rate $\tau_P$ of a compact Coxeter polyhedron $P\subset\mathbb H^n$ is either not a Salem number or satisfies $\tau_P>\tau_{[3,8]}$.
We distinguish between the two cases that (i) $P$ has mutually intersecting facets or that
(ii) the polyhedron $P$ has a pair of disjoint facets. 

In case (i), we know by Example \ref{ex4} that $P$ is of dimension four and equals either one of the five Coxeter simplices or one of the seven Esselmann polyhedra.  By Remark \ref{2Salem}, the growth rates of the five Coxeter $4$-simplices 
are not Salem numbers anymore. For the seven Esselmann polyhedra $E_i$,
the growth functions $f_i(t)\,,\,1\le i\le7\,,$ can be 
determined by means of Steinberg's formula; see also \cite[pp. 89-90]{Perren}.
This can also be achieved by using the software CoxIter~\cite{CoxIter}. Then, by analysing the sign changes of the denominator polynomials
of $f(t)$, one checks that the biggest real pole lies -- roughly -- between $1.90$ and $2.61$. The smallest growth rate among the Esselmann polyhedra
is $\approx 1.902812$ and belongs to the group with Coxeter symbol $[8,3,4,3,8]$.
Hence, for $1\le i\le 7$, the growth rate $\tau_i$ of $E_i$ satisfies $\tau_i>\tau_{[3,8]}$.

In case (ii), $P$ has a pair of disjoint facets. It follows that the Coxeter group $G=(W,S)$ of rank $N>5$ associated with $P$ has a Coxeter diagram $\Gamma_G$ with a dotted edge; see Section \ref{section2-2}.
Denote by $v,w$ two nodes connected by a dotted edge in $\Gamma_G$. The product of the corresponding generators in $S$ encoded by $v,w$ is of infinite order $\infty$. 
Consider the \emph{abstract} Coxeter graph $\Gamma$ of $(W,S)$ that results from $\Gamma_G$ by replacing all present bold and dotted edges by edges with weight $\infty$. The graph $\Gamma$ contains the edge of weight $\infty$ connecting the nodes $v,w$.
Since $\Gamma$ is connected of order $N>5$, there is a node $u$ in $\Gamma$ which is connected to one or both of the nodes $v,w$
by an edge with weight $2\le k\le\infty$ and with weight $3\le l\le\infty$, say. 
The nodes $u,v$ and $w$ determine a subgraph $\Gamma_{uvw}$ of $\Gamma$.
Hence we get the following sequence of ordered Coxeter graphs by taking into account Figure \ref{fig:subgraph1}.

%%%%%%%%%%%
\begin{figure}[H]
\centering
\begin{picture}(190,40)
    % Circles
    \put(0,15){\circle*{4}}
    \put(15,15){\circle*{4}}
    \put(30,15){\circle*{4}}
    % Line
    \put(0,15){\line(1,0){30}}
    
        % Circles
    \put(60,15){\circle*{4}}
    \put(75,15){\circle*{4}}
    \put(90,15){\circle*{4}}
    % Line
    \put(60,15){\line(1,0){30}}
    
            % Circles
    \put(130,22){\circle*{4}}
    \put(150,22){\circle*{4}}
    \put(140,7){\circle*{4}}
    % Line
    \put(130,22){\line(1,0){20}}
     \put(141,7){\line(3,5){10}}
     \put(139,7){\line(-3,5){10}}   
     
    \put(120,21){$\scriptstyle v$}
    \put(154,21){$\scriptstyle w$}
    \put(137,-2){$\scriptstyle u$}   
    \put(126,9){$\scriptstyle k$}
    \put(149,9){$\scriptstyle l$}   

 %   \put(-60,10){$\Gamma\,\,:$}
 
    \put(40,12){$\le$}     
    \put(100,12){$\le$} 
    \put(170,12){$\le$}    
    \put(185,12){$\Gamma$}      
       
    \put(20,18){$\scriptstyle 8$}  
    \put(78,19){$\scriptstyle \infty$}  
    \put(136,26){$\scriptstyle \infty$}        
\end{picture}
\caption{A sequence of ordered Coxeter graphs.}
\label{fig:subgraph2}
\end{figure}
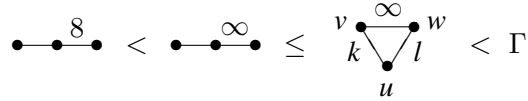
%%%%%%%%%%%%%%%%
By means of the growth rate monotonicity given by Lemma \ref{Terragni}, combined with Remark \ref{strict}, \eqref{eq:38}, we deduce that $\tau_{[3,8]}<\tau_{[3,\infty]}\le\tau_{G}$ as asserted.
\end{proof}
%\hfill\qedsymbol{}
\begin{remark}\label{higherdim-Salem}
The proof of Theorem~\ref{main01} shows that if the growth rate $\tau_P$ of a compact hyperbolic Coxeter polyhedron $P$ of dimension $n\ge4$
happens to be a Salem number, then $\tau_P>\tau_{[3,8]}$.  
However, there is no example known to date where this is the case.
\end{remark}
%%%%%%%%%%%%%%%%
%%%%%%%%%%%%%%%%%%
\section{The spectral radius of a Coxeter transformation}\label{spectral}
%%%%%%%%%%%%%
Consider an (abstract) Coxeter system~$(W,S)$ of rank $N$ with finite generating set~$S=\{s_1,\dots,s_N\}$ 
and associated Coxeter diagram $\Gamma$. Denote by $c\in W$ a Coxeter element and represent it geometrically 
by its Coxeter transformation $C=\rho(c)\in \mathrm{GL}_N(V)$.  Recall that 
all Coxeter elements are conjugate in $W$ if $\Gamma$ is a tree. In this case, the spectral radius $\lambda$ of $C$
is related to the leading eigenvalue $\alpha$ of the adjacency matrix of $\Gamma$ by $\,\alpha^2 = 2 + \lambda+\lambda^{-1}\,$; see Section~\ref{section2-1}.

Recall that if all the weights of the edges of a Coxeter diagram $\Gamma$
are equal to~$3$, then the Coxeter adjacency matrix equals the 
adjacency matrix of the underlying abstract graph. In this case, we will sometimes simply write ``graph" instead of ``Coxeter diagram". We will also write ``Coxeter tree" to stress the case where the underlying abstract graph of a Coxeter diagram is a tree.

We now define specific Coxeter diagrams that we use in this section.

Let~$k\ge3$ and $p_1,\dots,p_k\ge2$ be integers. 
Recall that the \emph{star graph}~$\mathrm{Star}(p_1,\dots,p_k)$
is defined to be the tree with one vertex of valency~$k$ that has~$k$ outgoing paths of respective lengths~$p_i -1$.
For example,~$\mathrm{Star}(2,3,7)$ the graph of the Coxeter group~$E_{10}$.

Suppose now that ~$\frac{1}{p_1}+\dots+\frac{1}{p_k}<k-2$, and consider the Coxeter diagram
$\Gamma(p_1,\ldots,p_k)$ of the compact hyperbolic Coxeter polygon $P=(p_1,\ldots,p_k)\subset\mathbb H^2$; see Example~\ref{ex1}. Recall that its growth rate $\tau$ is a Salem number. 

The goal of the section is to prove our second main result as stated by Theorem~\ref{main02} in Section~\ref{Intro}. 

For a Coxeter tree~$\Gamma$, we denote by~$C_\Gamma$ the associated Coxeter transformation and by~$\Phi_\Gamma (t)$ its characteristic polynomial. 
Given a graph~$\Gamma$ and a vertex~$v$ of~$\Gamma$, we denote by~$\Gamma-v$ the subgraph obtained by deleting~$v$ and all its adjacent edges from~$\Gamma$.
A~\emph{leaf} of a tree is a vertex of valency one.

\begin{lemma}
\label{Coxeterrecursion}
Let~$\Gamma$ be a Coxeter tree and let~$v$ be a leaf. Let~$v'$ be the unique vertex adjacent to~$v$, and let~$m$ be the weight of the edge connecting~$v$ and~$v'$. Then we have the following identity in~$\mathbb{R}[t]$:
\[
\Phi_\Gamma (t) = (1+t)\cdot\Phi_{\Gamma -v} - 4\cos^2\big(\frac{\pi}{m}\big)t\cdot\Phi_{\Gamma-v-v'}(t).
\]
\end{lemma}

\begin{proof}
Since the conjugacy class of the Coxeter transformation does not depend on the Coxeter element, we are free to choose the \emph{bipartite} Coxeter transformation for our calculations. More precisely, we partition the vertices of~$\Gamma$ into two sets~$V_1$ and~$V_2$ so that all edges of~$\Gamma$ connect a vertex from~$V_1$ 
with a vertex in~$V_2$. We then choose an ordering of the vertices of~$\Gamma$ so that all vertices in~$V_1$ appear before the vertices of~$V_2$, 
and take the corresponding Coxeter element. See Figure~\ref{H283} for an example of a bipartite ordering of the vertices of a tree.
Now, let~$\left(\begin{smallmatrix} 0 & X \\ X^\top & 0\end{smallmatrix}\right)$ be the adjacency matrix with respect to our chosen bipartite ordering of the vertices. 
For the Coxeter transformation, we then get that 
\[
C_\Gamma = \left(\begin{smallmatrix} -I & X \\ 0 & I\end{smallmatrix}\right)\left(\begin{smallmatrix} I & 0 \\ X^\top & -I\end{smallmatrix}\right)= 
- \left(\begin{smallmatrix} I - XX^\top & X \\ -X^\top & I\end{smallmatrix}\right) 
= - \left(\begin{smallmatrix} I & X \\ 0 & I\end{smallmatrix}\right)\left(\begin{smallmatrix} I & 0 \\ -X^\top & I\end{smallmatrix}\right).
\]
For the characteristic polynomial, we obtain
\begin{align}
\label{eq1}
\Phi_\Gamma(t) &= \det(t\cdot I - C_\Gamma) \\
&= \det(t\cdot I + \left(\begin{smallmatrix} I & X \\ 0 & I\end{smallmatrix}\right)\left(\begin{smallmatrix} I & 0 \\ -X^\top & I\end{smallmatrix}\right))\\
&= \det(t\cdot \left(\begin{smallmatrix} I & 0 \\ X^\top & I\end{smallmatrix}\right) + \left(\begin{smallmatrix} I & X \\ 0 & I\end{smallmatrix}\right)).\label{eq3}
\end{align}
The matrix we are taking the determinant of can be chosen to have the form
\[
  \renewcommand{\arraystretch}{1.2}
  \left(
  \begin{array}{ c c c | c c | c c c }
    \multicolumn{1}{c}{} & & & 0 & {0}&  & &  \\
    \multicolumn{1}{c}{} & (1+t)I & & {\vdots} & {\vdots} &  & \ast&  \\
     \multicolumn{1}{c}{} &  & & 0 & {0} &  & &  \\
    \cline{1-8}
    0 & \cdots & 0 & 1+t & 2\cos(\frac{\pi}{m}) & \ast &\cdots & \ast \\
    0 & \cdots & 0 & 2\cos(\frac{\pi}{m})t & 1+t & 0 &\cdots & 0 \\
   \cline{1-8}
     &  &  & \ast & 0 & & & \\
     & \ast &  & \vdots & \vdots & & (1+t)I & \\
     &  &  & \ast & 0 & & & 
  \end{array}
  \right),
\]
where the two middle columns and rows correspond to the vertices~$v$ and~$v'$. 
The assertion follows by developing the column and the row that correspond to the vertex~$v$.
\end{proof}

We now specialise Lemma~\ref{Coxeterrecursion} to the trees given by star graphs~$\mathrm{Star}(p_1,\dots,p_k)$.
In order to do so, we first simplify our notation as follows. Let~$C_{p_1,\dots,p_k}$ be the Coxeter transformation of~$\mathrm{Star}(p_1,\dots,p_k)$, 
and denote by~$\Phi_{p_1,\dots,p_k}(t)$ the characteristic polynomial of~$C_{p_1,\dots,p_k}$. 
The following statements follow directly from Lemma~\ref{Coxeterrecursion}. 

\begin{lemma}
\label{Coxeterskein}
Let~$p_1,\dots,p_k\ge2$ be integers.
\begin{enumerate}
\item If~$p_k\ge 4$, we have the following equality in~$\mathbb{Z}[t]$:
\[
\Phi_{p_1,\dots,p_k}(t) = (1+t)\cdot\Phi_{p_1,\dots,p_k-1}(t) -t\cdot\Phi_{p_1,\dots,p_k-2}(t).
\]
\item If~$p_k = 3$, we have the following equality in~$\mathbb{Z}[t]$:
\[
\Phi_{p_1,\dots,p_{k-1},3}(t) = (1+t)\cdot\Phi_{p_1,\dots,p_{k-1},2}(t) -t\cdot\Phi_{p_1,\dots,p_{k-1}}(t).
\]
\end{enumerate}
\end{lemma}

\begin{example}
\label{starexample}
The Coxeter transformation of the star graph with~$k\ge 1$ arms of length one has characteristic polynomial~$(t+1)^{k-1}(t^2-(k-2)t+1)$.
This can be verified, for example, by an inductive argument and Lemma~\ref{Coxeterrecursion}.
\end{example}

\begin{definition}
\label{denominatorpoly}
For integers~$p_1,\dots,p_k\ge2$, define
 $\Delta_{p_1,\dots,p_k}(t)\in\mathbb Z(t)$ by
\[
\Delta_{p_1,\dots,p_k}(t) := {[2]\prod_{i=1}^k [p_i]}\left( 1-\frac{k}{[2]} + \sum_{i=1}^k \frac{1}{[2][p_i]}\right)=[2]\prod_{i=1}^k [p_i] - k\prod_{i=1}^k [p_i] + \sum_{i=1}^k \prod_{j\ne i} [p_j]\,.
\]
\end{definition}

Recall that 
for~$k\ge3$ and $\frac{1}{p_1}+\dots+\frac{1}{p_k}<k-2$, the growth function $f_P(t)$ of a compact hyperbolic polygon $P=(p_1,\dots,p_k)$
is reciprocal and given by (see \eqref{eq:reciprocal}--\eqref{eq:Delta})
\[
f_P(t)=\frac{[2]\prod_{i=1}^k [p_i]}{\Delta_{p_1,\dots,p_k}(t)}\,.
\]

\begin{example}
\label{rightangledexample}
For $p_1=\dots=p_k=2$, 
we obtain
\[
\Delta_{2,\dots,2}(t) = [2]^{k+1} - k[2]^k + k[2]^{k-1} = (t+1)^{k-1}(t^2-(k-2)t+1).
\]
This polynomial equals the characteristic polynomial of the Coxeter transformation of the star graph with~$k\ge3$ arms of length one,
see Example~\ref{starexample}.
\end{example}

\begin{proposition}
\label{eq_prop} Let~$p_1,\dots,p_k\ge2$ be integers. 
We have the following equality in~$\mathbb{Z}[t]:$
\[ 
\Delta_{p_1,\dots,p_k}(t) = \Phi_{p_1,\dots,p_k}(t).
\]
\end{proposition}

In order to prove Proposition~\ref{eq_prop}, we establish recursion formulas for the polynomial~$\Delta_{p_1,\dots,p_k}(t)$. 
These recursion formulas have the same form as the recursion formulas we gave for the polynomial~$\Phi_{p_1,\dots,p_k}(t)$ in Lemma~\ref{Coxeterskein}. 
This is the content of the following lemma.

\begin{lemma}
\label{growthskein}
Let~$p_1,\dots,p_k\ge2$ be integers.
\begin{enumerate}
\item If~$p_k\ge 4$, we have the following equality in~$\mathbb{Z}[t]$: \[\Delta_{p_1,\dots,p_k}(t) = (1+t)\cdot\Delta_{p_1,\dots,p_k-1}(t) -t\cdot\Delta_{p_1,\dots,p_k-2}(t).\]
\item If~$p_k=3$, we have the following equality in~$\mathbb{Z}[t]$: \[\Delta_{p_1,\dots,p_{k-1},3}(t) = (1+t)\cdot\Delta_{p_1,\dots,p_{k-1},2}(t) -t\cdot\Delta_{p_1,\dots,p_{k-1}}(t).\]
\end{enumerate}
\end{lemma}

\begin{proof}
We have
\begin{align*}
\Delta_{p_1,\dots,p_k}(t) &= [2]\prod_{i=1}^k [p_i] - k\prod_{i=1}^k [p_i] + \sum_{i=1}^k \prod_{j\ne i} [p_j] \\
&= [p_k]\left([2]\prod_{i=1}^{k-1}[p_i] - k\prod_{i=1}^{k-1}[p_i] + \sum_{i=1}^{k-1}\prod_{j\ne i,k}[p_j] \right) + \prod_{i=1}^{k-1}[p_i]\\
&= [p_k]\left(\Delta_{p_1,\dots,p_{k-1}}(t) -\prod_{i=1}^{k-1}[p_i] \right)+  \prod_{i=1}^{k-1}[p_i]\\
&= [p_k]A + B,
\end{align*}
where the polynomials~$A = \left(\Delta_{p_1,\dots,p_{k-1}}(t) -\prod_{i=1}^{k-1}[p_i] \right)$ and~$B =  \prod_{i=1}^{k-1}[p_i]$ 
do not depend on~$p_k$. 
We now calculate
\begin{align*}
 \Delta_{p_1,\dots,p_k}(t) &= [p_k]A+B \\
&= \left((1+t)[p_k-1] -t[p_k-2]\right)A + B\\
&= (1+t)\left([p_k-1]A+ B \right)-t\left([p_k-2]A + B\right).
\end{align*}

If~$p_k\ge4$, the last line equals~$(1+t)\cdot\Delta_{p_1,\dots,p_k -1}(t) -t\cdot\Delta_{p_1,\dots,p_k -2}(t)$, which proves~(1). 
On the other hand, if~$p_k=3$, we have~$\left([p_k-2]A + B\right) = \Delta_{p_1,\dots,p_{k-1}}(t)$,
since~$[1] = 1$. This proves~(2).
\end{proof}

\begin{proof}[Proof of Proposition~\ref{eq_prop}]
First of all, we note that we can permute the~$p_i$ without changing~$\Delta_{p_1,\dots,p_k}(t)$ or~$\Phi_{p_1,\dots,p_k}(t)$.
Repeatedly using the recursion formulas~(1) of Lemma~\ref{Coxeterskein} and Lemma~\ref{growthskein},
the statement can hence be reduced to the class of cases where all~$p_i$ are in the set~$\{2,3\}$. Within this class, we now proceed by induction 
on the number~$N$ of~$p_i = 3$. 

We have two base cases,~$N=0$ and~$N=1$. If~$N = 0$, we are in the case~$k$ arbitrary and~$p_i = 2$ for all~$i$. 
In this case, we are done by the Examples~\ref{starexample} and~\ref{rightangledexample}.
If~$N = 1$, there are two possibilities to consider. If~$k=1$, then~$\Delta_3(t) = [4] = \Phi_3(t)$
is a straightforward verification. If~$k\ge2$, we use the recursion formulas~(2) of Lemma~\ref{Coxeterskein} and Lemma~\ref{growthskein}
to reduce the statement to the case~$N=0$.

For the inductive step, we assume that~$N\ge 2$. As before, we use the recursion formulas~(2) of Lemma~\ref{Coxeterskein} and Lemma~\ref{growthskein}
to reduce the statement to the case~$N-1$. This finishes the proof.
\end{proof}

We are now ready to prove our second main result.

\begin{proof}[Proof of Theorem~\ref{main02}]
By definition, the spectral radius of the Coxeter transformation of the star graph~$\mathrm{Star}(p_1,\dots,p_k)$ 
equals the absolute value of the largest root of~$\Phi_{p_1,\dots,p_k}(t)$. 
By Proposition~\ref{eq_prop}, this in turn equals the absolute value of the largest root of~$\Delta_{p_1,\dots,p_k}(t)$. 

For $\frac{1}{p_1}+\dots+\frac{1}{p_k}<k-2$, consider the growth series $f(t)$ of the compact hyperbolic Coxeter polygon $P=(p_1,\dots,p_k)$. Since the denominator polynomial of the rational function $f(t)$ equals~$\Delta_{p_1,\dots,p_k}(t)$,
the growth rate $\tau$ of $P$, as given by the inverse of the radius of convergence of $f(t)$, also equals the absolute value of the largest root of~$\Delta_{p_1,\dots,p_k}(t)$.
\end{proof}

\newpage

\section{The tetrahedral groups $[3,5,3]$ and $[4,3,5]$}\label{growthspectral}

In the previous section, we have shown that the growth rates of planar hyperbolic Coxeter groups are 
spectral radii of Coxeter transformations. In this section, we start the investigation of this property in dimension three.
By giving both an example of a growth rate that is the spectral radius of a Coxeter transformation and an example that 
is not, we illustrate that the question becomes more difficult. 

\begin{definition}
For integers~$i,k\ge2$ and~$j\ge1$, let~$H(i,j,k)$ be the connected tree with two vertices~$v_1,v_2$ of valency three that are connected by a path of length~$j$. 
Furthermore,~$v_1$ has two additional outgoing paths: one of length~$i-1$ and one of length~$1$. Similarly,~$v_2$ has two 
additional outgoing paths: one of length~$k-1$ and one of length~$1$. 
\end{definition}
For example, Figure~\ref{H283} depicts the graph~$H(2,8,3)$. 

%\begin{example}
Consider the compact Coxeter tetrahedron~$[4,3,5]$ (see also Remark \ref{435}). The associated growth rate $\tau_{[4,3,5]}$ is the Salem number with 
minimal polynomial
\[
p(t) = t^8 - t^7 + t^6 - 2t^5 + t^4 - 2t^3 + t^2 -t + 1.
\]
This follows from work of Parry \cite{Parry} and can be conveniently verified by means of
the software CoxIter~\cite{CoxIter}. 
In particular, the growth rate $\tau_{[4,3,5]}$ 
is the largest root of~$p(t)$ and is $\approx1.359999$. 
\begin{figure}[h]
\begin{center}
\def\svgwidth{330pt}
\begingroup%
  \makeatletter%
  \providecommand\color[2][]{%
    \errmessage{(Inkscape) Color is used for the text in Inkscape, but the package 'color.sty' is not loaded}%
    \renewcommand\color[2][]{}%
  }%
  \providecommand\transparent[1]{%
    \errmessage{(Inkscape) Transparency is used (non-zero) for the text in Inkscape, but the package 'transparent.sty' is not loaded}%
    \renewcommand\transparent[1]{}%
  }%
  \providecommand\rotatebox[2]{#2}%
  \ifx\svgwidth\undefined%
    \setlength{\unitlength}{515.17781896bp}%
    \ifx\svgscale\undefined%
      \relax%
    \else%
      \setlength{\unitlength}{\unitlength * \real{\svgscale}}%
    \fi%
  \else%
    \setlength{\unitlength}{\svgwidth}%
  \fi%
  \global\let\svgwidth\undefined%
  \global\let\svgscale\undefined%
  \makeatother%
  \begin{picture}(1,0.23014704)%
    \put(0,0){\includegraphics[width=\unitlength,page=1]{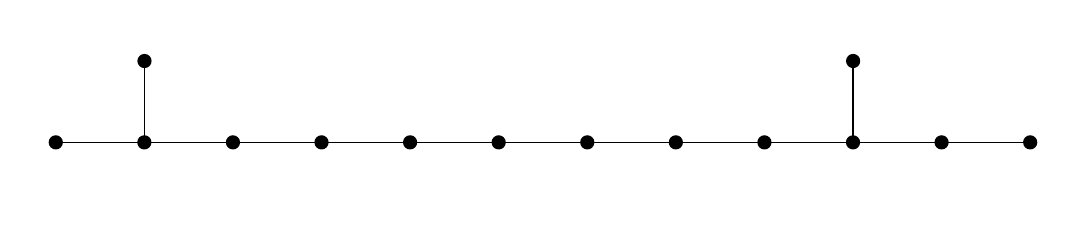}}%
    \put(0.12494773,0.04669121){\color[rgb]{0,0,0}\makebox(0,0)[lb]{\smash{$1$}}}%
    \put(0.289714,0.04669121){\color[rgb]{0,0,0}\makebox(0,0)[lb]{\smash{$2$}}}%
    \put(0.20839672,0.04669121){\color[rgb]{0,0,0}\makebox(0,0)[lb]{\smash{$9$}}}%
    \put(0.0425386,0.04669121){\color[rgb]{0,0,0}\makebox(0,0)[lb]{\smash{$7$}}}%
    \put(0.45479222,0.04669121){\color[rgb]{0,0,0}\makebox(0,0)[lb]{\smash{$3$}}}%
    \put(0.62247009,0.04669121){\color[rgb]{0,0,0}\makebox(0,0)[lb]{\smash{$4$}}}%
    \put(0.52950638,0.04669121){\color[rgb]{0,0,0}\makebox(0,0)[lb]{\smash{$11$}}}%
    \put(0.36364825,0.04669121){\color[rgb]{0,0,0}\makebox(0,0)[lb]{\smash{$10$}}}%
    \put(0.78838023,0.04669121){\color[rgb]{0,0,0}\makebox(0,0)[lb]{\smash{$5$}}}%
    \put(0.95314656,0.04669121){\color[rgb]{0,0,0}\makebox(0,0)[lb]{\smash{$6$}}}%
    \put(0.86018298,0.04669121){\color[rgb]{0,0,0}\makebox(0,0)[lb]{\smash{$14$}}}%
    \put(0.69432473,0.04669121){\color[rgb]{0,0,0}\makebox(0,0)[lb]{\smash{$12$}}}%
    \put(0.81791243,0.16523555){\color[rgb]{0,0,0}\makebox(0,0)[lb]{\smash{$13$}}}%
    \put(0.15863928,0.16523555){\color[rgb]{0,0,0}\makebox(0,0)[lb]{\smash{$8$}}}%
    \put(0,0){\includegraphics[width=\unitlength,page=2]{H283.pdf}}%
  \end{picture}%
\endgroup%
\caption{The graph~$H(2,8,3)$ with a bipartite ordering of its vertices.}
\label{H283} 
\end{center}
\end{figure}
This number equals the spectral radius of the Coxeter transformation associated with the graph~$H(2,8,3)$, depicted in Figure~\ref{H283}. 
In fact, using Equations~(\ref{eq1})-(\ref{eq3}) from the proof of Lemma~\ref{Coxeterrecursion}, one 
can compute the characteristic polynomial of the Coxeter transformation of~$H(2,8,3)$ as the determinant of a block matrix involving 
identity matrices and the matrix~$X$ which is defined via a bipartite adjacency matrix of the graph. In the case of~$H(2,8,3)$, with the  
numbering of the vertices indicated in Figure~\ref{H283}, the matrix~$X$ becomes
\[
\begin{pmatrix}
1 & 1 & 1 & 0 & 0 & 0 & 0 & 0\\
0 & 0 & 1 & 1 & 0 & 0 & 0 & 0\\
0 & 0 & 0 & 1 & 1 & 0 & 0 & 0\\
0 & 0 & 0 & 0 & 1 & 1 & 0 & 0\\
0 & 0 & 0 & 0 & 0 & 1 & 1 & 1\\
0 & 0 & 0 & 0 & 0 & 0 & 0 & 1
\end{pmatrix}.
\]
A straightforward computation yields that the characteristic polynomial of the Coxeter transformation is given by 
\[
t^{14} + t^{13} - t^{12} - 2t^{11} - t^{10} - t^4 - 2t^3 - t^2 + t + 1,
\]
which factors as
\[
(t^8 - t^7 + t^6 - 2t^5 + t^4 - 2t^3 + t^2 - t + 1)(t^4 - t^2 + 1)(t + 1)^2.
\]
In particular, we see that the spectral radius of the Coxeter transformation associated with~$H(2,8,3)$ equals 
the Salem number with minimal polynomial~$p(t)$. 
%\end{example}

\begin{proposition}
\label{353nosr}
The growth rate of the tetrahedral group~$[3,5,3]$ is not equal to the spectral radius of a Coxeter transformation. 
\end{proposition}

The growth rate $\tau_{[3,5,3]}$ of the compact Coxeter tetrahedron~$[3,5,3]$ is equal to the Salem 
number~$\lambda_0\approx1.350980$ with minimal polynomial
\[
t^{10}-t^9-t^6+t^5-t^4-t+1,
\]
(see Section \ref{section3-1}). Our proof of Proposition~\ref{353nosr} is based on 
McMullen's classification of minimal hyperbolic Coxeter systems (see Section \ref{section3-1}). 
Firstly, we note that by Proposition~7.5 of McMullen~\cite{McM02}, if a Coxeter transformation has spectral 
radius smaller than the golden ratio $(1+\sqrt{5})/2$, then it is the Coxeter transformation associated with a Coxeter diagram whose underlying abstract graph is a tree, with no 
restriction on edge weights. 
We need a 
slightly stronger statement, given by the following lemma.

\begin{lemma}
\label{treesr}
If~$\lambda<1.35999$ is the spectral radius of a Coxeter transformation, then~$\lambda$
is also the spectral radius of a Coxeter transformation of a tree with constant edge weights all equal to~$3$.
\end{lemma}

\begin{proof}
By Proposition~7.5 of McMullen~\cite{McM02}, we know that~$\lambda$ must be the spectral radius of a Coxeter transformation of a tree.
We now want to show that we can assume the tree to have constant edge weights all equal to~$3$.
To this end we assume that~$\lambda<1.35999$ is the spectral radius of a Coxeter transformation of a tree~$\Gamma$ with at least one edge weight~$\ge4$. 
By McMullen's classification of the 38 minimal hyperbolic Coxeter diagrams~\cite{McM02}, the Coxeter tree~$\Gamma$ must dominate 
either the~$\mathrm{Star}(2,4,5)$ or the~$\mathrm{Star}(2,3,7)$. Indeed, the Coxeter transformation of all the other minimal hyperbolic Coxeter trees have larger spectral radii. 

Now, at least one edge of~$\Gamma$ must have weight~$\ge4$. The only possibility for this to happen is if the weight is exactly~$4$ and the edge weighted~$4$ is adjacent to a leaf of~$\Gamma$. Indeed, in all other cases, a minimal diagram given in 
\cite[Table~5]{McM02} other than~$\mathrm{Star}(2,4,5)$ or~$\mathrm{Star}(2,3,7)$ would be 
dominated by~$\Gamma$, and hence the spectral radius would have to be larger than~$1.35999$. 

The result now follows from the following observation: a leaf~$v$ that is connected to a vertex~$w$ by an edge 
of weight~$4$ can be replaced by two leaves~$v_1$ and~$v_2$ that are both connected to~$w$ by an edge of weight~$3$, without changing the spectral radius of the 
adjacency matrix. Hence the spectral radius of the Coxeter transformation does not change by this replacement, since it is uniquely determined by the spectral radius of the adjacency matrix.
Let~$\Gamma'$ be the result of this replacement, and assume the vertices~$w$ and~$v$ are the two last ones with respect to the numbering for the adjacency matrix.
Then it can be verified directly that if~$(v_1, \dots, v_r, x, y)^\top$ is the Perron-Frobenius eigenvector of the adjacency matrix of~$\Gamma$, 
the vector~$(v_1,\dots,v_r,x,\frac{y}{\sqrt{2}},\frac{y}{\sqrt{2}})^\top$ is the Perron-Frobenius eigenvector of the adjacency matrix of~$\Gamma'$ and the two Perron-Frobenius eigenvalues agree. 
In particular, the spectral radii of the adjacency matrices of~$\Gamma$ and~$\Gamma'$ agree. 
\end{proof}

The second ingredient we need for the proof of Proposition~\ref{353nosr} is purely graph-theoretical and follows from the classification of trees whose adjacency matrices 
have small spectral radii. We stress that we deal with \emph{graph-theoretical} adjacency matrices here, that is, all coefficients are nonnegative integers. 

\begin{lemma}
\label{notree}
Let~$\alpha_0>2$ be defined by~$\alpha_0^2 = \lambda_0 + \lambda_0^{-1} +2$, where~$\lambda_0 \approx 1.350980$ is the growth rate of the Coxeter tetrahedron~$[3,5,3]$. Then~$\alpha_0$ is not the spectral radius of an adjacency matrix of a graph.
\end{lemma}

\begin{proof}
We note that~$\alpha_0 \approx 2.0226674 < \sqrt{2+\sqrt{5}}$ and use the classification of graphs with spectral radius 
smaller than~$\sqrt{2+\sqrt{5}}$, due to Brouwer and Neumaier~\cite{Brouwer89}. This classification states that the
graphs whose adjacency matrices have spectral radii strictly in between~$2$ and~$\sqrt{2+\sqrt{5}}$ are the following:
\begin{enumerate}
\item $\mathrm{Star}(p,q,r)$ where~$(p,q,r)$ is among
	\begin{enumerate}
	\item $(2,3,r)$ with~$r\ge7$,
	\item $(2,4,r)$ with~$r\ge5$,
	\item $(2,q,r)$ with~$q\ge r\ge5$,
	\item $(3,3,r)$ with~$r\ge 4$,
	\item $(3,4,4)$. 
	\end{enumerate}
\item $H(i,j,k)$ where~$(i,j,k)$ is among
	\begin{enumerate}
	\item $(i,j,k)$ with~$j\ge i+k$,
	\item $(3,j,k)$ with~$j\ge k+2$,
	\item $(2,j,k)$ with~$j\ge k-1$,
	\item $(2,1,3)$, $(3,4,3)$, $(3,5,4)$, $(4,7,4)$ or $(4,8,5)$. 
	\end{enumerate}
\end{enumerate}

To finish the proof, we use the values of spectral radii of adjacency matrices depicted in Table~\ref{tab:sr}. 

\renewcommand{\arraystretch}{1.2}
\begin{table}[h]
  \centering
  \begin{tabular}{c|c}
    Graph & approx.\ spectral radius of the adjacency matrix \\
    \hline
    $\mathrm{Star}(2,4,5)$ & 2.0153161  \\
    $\mathrm{Star}(2,4,6)$ & 2.0236833  \\
    $\mathrm{Star}(2,5,5)$ & 2.0285235  \\
    $\mathrm{Star}(3,3,4)$  & 2.0285235  \\
    $H(2,9,3)$ & 2.0227871   \\
    $H(2,10,3)$ & 2.0220988  \\
    $H(3,20,3)$ & 2.0227871 \\
    $H(3,21,3)$ & 2.0224205
  \end{tabular}
  \caption{Some approximate spectral radii of graphs.}
  \label{tab:sr}
\end{table}
We first deal with the cases of the star graphs, and use that the spectral radii of adjacency matrices 
are monotonic with respect to taking subgraphs, see, for example, Proposition~3.1.1 in~\cite{BrHa}. 

The stars $(3,4,4)$ and $(3,3,r)$ with~$r\ge 4$ have $(3,3,4)$ as a subgraph, and hence the spectral radius of their adjacency matrix is~$> \alpha_0$.
This deals with the cases 1(d) and 1(e). 

The stars $(2,q,r)$ with~$q\ge r\ge5$ have $(2,5,5)$ as a subgraph and hence the spectral radius of their adjacency matrix is~$> \alpha_0$. This deals with case 1(c). 

The star $(2,4,5)$ has spectral radius~$<\alpha_0$. Furthermore, all stars $(2,4,r)$ for~$r\ge6$ have~$(2,4,6)$ as a subgraph and hence the spectral radius of 
their adjacency matrix is~$>\alpha_0$. This deals with the case 1(b). 

In order to treat the case 1(a), we note that the sequence of spectral radii of the adjacency matrices of the graphs~$H(2,j,3)$ is monotonically decreasing in~$j$. 
This follows from the fact that subdividing an edge that does not lie on an endpath does not increase the spectral radius, see, for example, Proposition 3.1.4 in~\cite{BrHa}. 
From Table~\ref{tab:sr} 
we obtain that~$H(2,10,3)$ is smaller than~$\alpha_0$ and hence so is the spectral radius of~$H(2,j,3)$ for all~$j\ge10$. In particular, since every star of the type 
$(2,3,r)$ is the subgraph of a graph~$H(2,j,3)$ for~$j$ large enough, also the spectral radii of the stars~$(2,3,r)$ are smaller than~$\alpha_0$.

We now deal with the graphs of type~$H(i,j,k)$. 

As soon as~$i$ or~$k$ is~$\ge 4$ and~$j\ne3$, it follows by the classification that the star of type~$(2,4,6)$ is a subgraph. 
Hence the spectral radius is~$>\alpha_0$, which must also be the case for~$j=3$, since the spectral radius is monotonically decreasing in~$j$. 
This eliminates many cases. Up to graph isomorphism, the only cases that we still have to consider are the following ones. 

\begin{enumerate}
\item[(i)] $H(2,j,3)$ where~$j\ge1$,
\item[(ii)] $H(3,j,3)$ where~$j\ge 4$. 
\end{enumerate}

In both cases, the spectral radius is again a decreasing sequence in the parameter~$j$. Hence, in both cases the values given in Table~\ref{tab:sr} suffice 
to exclude that a spectral radius of the adjacency matrix of a graph of type~$H(i,j,k)$ equals~$\alpha_0$. 
This concludes the proof.
\end{proof}

We are now ready to prove Proposition~\ref{353nosr}.

\begin{proof}[Proof of Proposition~\ref{353nosr}]
We want to show that~$\lambda_0\approx1.350980$ is not the spectral radius of a Coxeter transformation. 
By Lemma~\ref{treesr}, if~$\lambda_0$ was the spectral radius of a Coxeter transformation, then there would exist 
a tree~$\Gamma$ with constant edge weights all equal to~$3$ and such that~$\lambda_0$ is the spectral radius 
of the Coxeter transformation associated with~$\Gamma$. 
The spectral radius~$\lambda$ of the Coxeter transformation is related to the spectral radius~$\alpha$ 
of the adjacency matrix of~$\Gamma$ by the identity \eqref{eq:eigen-spectral}.
We note that in the case of constant edge weights equal to 3, the Coxeter adjacency matrix equals the graph-theoretic adjacency matrix of the tree.
But then, by Lemma~\ref{notree}, the equation \eqref{eq:eigen-spectral}
does not have a solution among trees for~$\lambda=\lambda_0$. 
It follows that~$\lambda_0$  cannot be the spectral radius of a Coxeter transformation. 
\end{proof}
%%%%%%%%%%%%%%%%%%%%%%%%%%%%%%%%%%

%%%%%%%%%%%%%%%%%%%%%%%%%%%
\end{document}